\documentclass[11pt,reqno]{amsart}

\usepackage{amssymb}
\usepackage{amscd}
\usepackage{amsfonts}
\usepackage{mathrsfs}
\usepackage{setspace}
\usepackage{version}
\usepackage{mathtools}
\usepackage{enumerate}
\usepackage[english]{babel}

\numberwithin{equation}{section} 
\theoremstyle{plain}
\newtheorem{theorem}[subsection]{Theorem}
\newtheorem{lemma}[subsection]{Lemma}
\newtheorem{definition}[subsection]{Definition}
\newtheorem{remark}[subsection]{Remark}
\newtheorem{proposition}[subsection]{Proposition}

\newcommand\id{\operatorname{id}}

\begin{document}

\title{A Fenchel-Moreau theorem for $\bar L^0$-valued functions}

\author{Samuel Drapeau}
\address{Shanghai Jiao Tong University, Shanghai Advanced Institute of Finance}
\email{sdrapeau@saif.sjtu.edu.cn}

\author{Asgar Jamneshan}
\address{Department of Mathematics and Statistics, University of Konstanz}
\email{asgar.jamneshan@uni-konstanz.de}

\author{Michael Kupper}
\address{Department of Mathematics and Statistics, University of Konstanz}
\email{kupper@uni-konstanz.de}

\thanks{We thank an anonymous referee for helpful comments. S.~D.~is supported by NSF of China grant ''Research Fund for International Young Scientists'' 11550110184. The grant "Assessment of Risk and Uncertainty in Finance" number AF0710020 from Shanghai Jiao Tong University is gratefully acknowledged. A.~J.~and M.~K.~are supported by DFG grant KU-2740/2-1.}


\begin{abstract}  
We establish a Fenchel-Moreau theorem for proper convex functions $f\colon X\to \bar{L}^0$, where $(X,Y,\langle \cdot,\cdot \rangle)$ is a dual pair of Banach spaces and $\bar L^0$ is the space of all extended real-valued  functions on a $\sigma$-finite measure space. 
We introduce the concept of stable lower semi-continuity which is shown to be equivalent to the existence of a dual representation $f(x)=\sup_{y \in L^0(Y)} \left\{\langle x, y \rangle - f^\ast(y)\right\}$, %\quad x\in X, 
where $L^0(Y)$ is the space of all strongly measurable functions with values in $Y$, and $\langle \cdot,\cdot \rangle$ is understood pointwise almost everywhere. 
The proof is based on a conditional extension result and conditional functional analysis. \\

\end{abstract}

\maketitle

\setcounter{tocdepth}{1}

\section{Introduction}

This article contributes to vector duality by providing a notion of lower semi-continuity and proving its equivalence to a Fenchel-Moreau type dual representation. 
Let $(\Omega,\mathcal{F},\mu)$ be a $\sigma$-finite measure space, $(X,Y,\langle \cdot, \cdot\rangle)$ a dual pair of Banach spaces and $\bar L^0$ the collection of all measurable functions $x\colon\Omega\to \mathbb{R}\cup\{\pm\infty\}$, where two of them are identified if they agree almost everywhere. 
Consider on $\bar L^0$ the order of almost everywhere dominance. 
Let $f\colon X\to \bar L^0$ be a proper convex function. 
We prove that stable lower semi-continuity (see below) is equivalent to the Fenchel-Moreau type dual representation 
\begin{equation}\label{darstellung}
 f(x)=\sup_{y\in L^0(Y)}\{\langle x,y\rangle - f^\ast(y)\}, \quad x\in X, 
\end{equation}
 where $L^0(Y)$ is the space of all strongly measurable functions $y\colon \Omega\to Y$ modulo almost everywhere equality, $f^\ast(\cdot)=\sup_{x\in X}\{\langle x,\cdot\rangle - f(x)\}$ is the convex conjugate and $\langle x,y\rangle(\omega):=\langle x,y(\omega)\rangle$ almost everywhere. 

The idea is to extend the algebraic and topological structure of $f\colon X\to \bar L^0$ to a larger $L^0$-module context in such a way that a conditional version of the Fenchel-Moreau theorem can be applied. 
More precisely, we first extend the duality pairing $\langle\cdot,\cdot\rangle$ to a conditional duality pairing on $L^0(X)\times L^0(Y)$. 
We consider on $L^0(X)$ the stable weak topology $\sigma_s(L^0(X),L^0(Y))$ which can be viewed as the conditional analogue of the weak topology $\sigma(X,Y)$. 
For a discussion of topologies in conditional settings or $L^0$-modules, we refer to \cite{DJKK13,kupper03,guo10,JZ2017compact}. 
We call a function $f:X\to \bar L^0$ $\sigma_s$-lower semi-continuous if its extension $f_s$ to step functions given by $f_s(\sum_k x_k 1_{A_k}):=\sum_k f(x_k) 1_{A_k}$ is lower semi-continuous w.r.t.~the relative $\sigma_s(L^0(X),L^0(Y))$-topology (notice that the space $L^0_s(X)$ of step functions with values in $X$ is a subset of $L^0(X)$). 
We prove that  $\sigma_s$-lower semi-continuity is sufficient to extend $f\colon X\to \bar L^0$ to a stable proper $L^0$-convex  and $\sigma_s(L^0(X),L^0(Y))$-lower semi-continuous function $F:L^0(X)\to \bar L^0$. 
Building on a conditional version of the Fenchel-Moreau theorem, we find the conditional dual representation 
\begin{equation}\label{darstellung1}
F(x)=\sup_{y\in L^0(Y)}\{\langle x,y\rangle - F^\ast(y)\}, \quad x\in L^0(X),  
\end{equation}
for a conditional convex conjugate $F^\ast\colon L^0(Y)\to \bar L^0$. 
Finally, by restricting \eqref{darstellung1} to $X$, we derive at the representation \eqref{darstellung}. 

In optimization Fenchel-Moreau duality is an important result for strong duality and related regularity conditions, see \cite{grad2017conditional} for vector optimization results based on the Fenchel-Moreau duality in this work. 
Our Fenchel-Moreau theorem cannot be obtained from scalarization techniques \cite{ioan2009duality,bot2011duality}, set-valued methods \cite{hamel09,hamel14,loehne2007duality} or vector-space techniques \cite{zowe75,koshifenchel1983}. 
The module approach in \cite{kutateladze81,kupper03} cannot be applied since a Banach space is a priori not an $L^0$-module. 
A similar approach to ours is taken in \cite{kutateladze06} with the tools of Boolean-valued analysis \cite{kusraev2012boolean}, albeit in the context of norm topologies. 
For further results in vector and conditional duality, we refer to \cite{ioan2009duality,jahn2004vector,heyde2008geometric,OZ2017stabil,zapata2016eberlein}.  

The remainder of this article is organized as follows. 
In Section \ref{s:extension} we introduce the setting and prove the main extension result. 
In Section \ref{s:main} we derive a vector-valued Fenchel-Moreau theorem. %and provide a simple application in finite dimensional vector duality. 

\section{Extension of  stable lower semi-continuous functions}\label{s:extension}

\subsection{Preliminaries}
Let $L^0$, $L^0_{++}$ and $\bar L^0$ denote the spaces of all measurable functions on a $\sigma$-finite measure space $(\Omega,\mathcal{F},\mu)$ with values in $\mathbb{R}$, $\mathbb{R}_{++}$ and $[-\infty,+\infty]$, where two of
them are identified if they agree almost everywhere (a.e.). In particular, all equalities and inequalities in $\bar L^0$ are understood in the a.e.~sense. Every nonempty subset $C$ of $\bar L^0$ has a least upper bound $\sup C:=\mathop{\rm ess\,sup} C$ and a greatest lower bound  $\inf C:=\mathop{\rm ess\,inf} C$
in $\bar L^0$ with respect to the a.e.~order. 

Throughout all functions on $\Omega$ are assumed to be (strongly) measurable and we identify functions which agree a.e..
Given a set $Z$, we denote by $L^0_s(Z)$ the space of all \emph{step functions} $\sum_{k} z_k 1_{A_k}:\Omega\to Z$, where $(z_k)$ is a sequence in $Z$, $(A_k)$ is a partition  of $\Omega$, and $\sum_{k} z_k 1_{A_k}$ denotes the function which is equal to $z_k$
for almost all $\omega\in A_k$. If $Z$ is partially ordered we consider on $L^0_s(Z)$ the partial order 
$\sum_{k} x_k 1_{A_k}\ge \sum_{l} y_l 1_{B_l}$ whenever $x_k\ge y_l$ for all $k,l$ with $\mu(A_k\cap B_l)>0$. Given a function $f$ from $Z$ to a set $\tilde Z$, its \emph{extension to step functions} $f_s:L^0_s(Z)\to L^0_s(\tilde Z)$ is defined by
 \[
 f\Big(\sum_k z_k 1_{A_k}\Big):=\sum_k f(z_k) 1_{A_k}.
 \]

A set $H$ of functions on $\Omega$ is called \emph{stable} (under countable concatenations) if it is non-empty and 
$\sum_{k} h_k 1_{A_k}\in H$ for every sequence $(h_k)$ in $H$ and every partition $(A_k)$ of $\Omega$. 
A \emph{stable family} $(h_i)_{i\in I}$ in $H$ is a family $(h_i)$ in $H$ indexed by a stable set $I$ of functions on $\Omega$ such that 
\[\sum_k h_{i_k} 1_{A_k}=h_{\sum_k i_k 1_{A_k}}\] 
for every sequence $(i_k)$ in $I$ and every partition $(A_k)$ of $\Omega$. A \emph{stable net} $(h_\alpha)$ in $H$ is a stable family indexed by a stable set of functions with values in a directed set. 
If the directed set is $\mathbb{N}$ the stable net $(h_n)$ is a \emph{stable sequence} in which case the index set equals $L^0_s(\mathbb{N})$. 
A stable family $(h_m)$ is a \emph{stable finite family} if it is indexed by a stable set of the form $\{m\in L^0_s(\mathbb{N}):1\le m\le n\}$ for some $n\in L^0_s(\mathbb{N})$.
Let $I$ and  $(H_i)$, $i\in I$, be stable sets of functions on $\Omega$. Then  $(H_i)_{i\in I}$ is called a \emph{stable family of stable sets} if 
\[\sum_k H_{i_k} 1_{A_k}:=\Big\{\sum_k h_{i_k} 1_{A_k}:h_{i_k}\in H_{i_k}\Big\}=H_{\sum_k i_k 1_{A_k}}\] 
for every sequence $(i_k)$ in $I$ and every partition $(A_k)$ of $\Omega$.
\begin{remark}\label{l:choice}
Given a stable family of stable sets $(H_i)_{i\in I}$ there exists a stable family $(h_i)_{i\in I}$ such that $h_i\in H_i$ for all $i\in I$. This follows by the same arguments as in \cite[Theorem 2.26]{DJKK13}, where the statement is shown within conditional set theory.
\end{remark}

\subsection{Stable lower semi-continuity} 
Let $(X,Y,\langle\cdot,\cdot \rangle)$ be a dual pair of Banach spaces such that
\begin{itemize}
\item[(i)] $|\langle x,y\rangle|\leq \|x\| \|y\|$ for all $x\in X$ and $y\in Y$, and 
\item[(ii)] both norm-closed unit balls are weakly closed.
\end{itemize} 

\begin{examples}
    \begin{enumerate}[a)]
        \item Let $X$ be a Banach space, $Y$ its topological dual space endowed with the operator norm and $\langle x, y\rangle:=y(x)$. 
        \item Let $X=L^p$ and $Y=L^q$ on a finite measure space $(S,\mathcal{S},\nu)$ with $1\leq p,q\leq \infty$ and $1/p+1/q\leq 1$ and $\langle f,g\rangle:=\int_S f g d\nu$. 
        \item Let $X=B_b(S)$ be the Banach space of bounded measurable functions on a measurable space $(S,\mathcal{S})$ endowed with the supremum norm, $Y=\mathcal{M}(S)$ the Banach space of finite signed measures endowed with the total variation norm and $ \langle f, \mu \rangle:=\int_S f d\nu$. 
        \item Let $X=C_b(S)$ be the Banach space of bounded continuous functions on a completely regular Hausdorff space endowed with the supremum norm, $Y=\mathcal{M}_r(S)$ the Banach space of finite signed inner regular measures on the Borel $\sigma$-algebra of $S$ endowed with the total variation norm and $ \langle f, \mu \rangle:=\int_S f d\nu$. 
\end{enumerate}
\end{examples} 

We denote by $L^0(X)$ and $L^0(Y)$ the spaces of all strongly measurable functions on $\Omega$ with values in $X$ and $Y$. 
We understand $X$ as a subset of $L^0_s(X)\subset L^0(X)$ via the embedding $x\mapsto x1_\Omega$. 
Recall that the norm of $X$ extends to $L^0(X)$ by $\|x\|:=\lim_{n\to \infty} \|x_n\|\in L^0$, where $(x_n)$ is a sequence 
in $L^0_s(X)$ such that $x_n\to x$ a.e.. In particular, for every $x\in L^0(X)$ and each $r\in L^0_{++}$ there exists $\tilde{x}\in L^0_s(X)$ such that $\|x-\tilde{x}\|\leq r$. 
Similarly, the duality pairing $\langle\cdot,\cdot \rangle$ can be extended from 
$X\times Y$ to $L^0(X)\times L^0(Y)$ by setting $\langle x, y\rangle:=\lim_{n\to \infty} \langle x_n,y_n\rangle\in L^0$,
where $(x_n)$ is a sequence in $L^0_s(X)$ such that $x_n\to x$ a.e.~and  $(y_n)$ is a sequence 
in $L^0_s(Y)$ such that $y_n\to y$ a.e.. 
Observe that 
\begin{equation}\label{eq:cauchyschwartz}
|\langle x,y\rangle|\leq \|x\|\|y\|
\end{equation}
for all $x\in L^0(X)$ and $y\in L^0(Y)$. 

Next, we endow $L^0(X)$ with a topological structure given by the following neighborhood base. 
For every $x\in L^0(X)$, define
\begin{equation*}
\mathcal{V}(x):=\left\{V^{r}_{(y_m)_{1\leq m\leq n}}(x) \colon r \in L^0_{++},\, (y_m)_{1\leq m\leq n} \text{ stable finite family in } L^0(Y)\right\}, 
\end{equation*}
where
\begin{equation*}
V^{r}_{(y_m)_{1\leq m \leq n}}(x):=\{\tilde{x} \in L^0(X) \colon |\langle \tilde{x} - x, y_m\rangle | \leq r \text{ for all } 1\leq m\leq n\}. 
\end{equation*}
We notice that $\mathcal{V}(x)$ is a stable family of stable sets. Indeed, for each $r \in L^0_{++}$ and 
every stable finite family $(y_m)_{1\leq m\leq n}$ in $L^0(Y)$ the set $V^{r}_{(y_m)_{1\leq m\leq n}}(x)$ in $L^0(X)$ is stable.   For every $n=\sum_j n_j 1_{A_j}\in L^0_s(\mathbb{N})$, $r\in L^0_{++}$ and $y_1,y_2, \dots \in L^0(Y)$ the stable set $V^{r}_{(y_m)_{1\leq m \leq n}}(x)$ is determined through the element $\sum_j (r,n_j,y_1,y_2,\dots,y_{n_j},0,0,\dots)1_{A_j}$ in the space of all strongly measurable functions on $\Omega$ with values
in the Banach space of sequences $(r,n,y_1,y_2,\dots)\in\mathbb{R}\times\mathbb{R}\times l^\infty(Y)$ with finite norm $\|(r,n,y_1,y_2,\dots)\|:=|r|+|n|+\sup_j \|y_j\|$.
Denoting by $I$ the stable set of all $\sum_j (r,n_j,y_1,y_2,\dots,y_{n_j},0,0,\dots)1_{A_j}$
for $n=\sum_j n_j 1_{A_j}\in L^0_s(\mathbb{N})$, $r\in L^0_{++}$ and $y_1,y_2,\dots\in L^0(Y)$,
one has $\mathcal{V}(x)=(V_i(x))_{i\in I}$ where $V_i(x)=V^{r}_{(y_m)_{1\leq m \leq n}}(x)$
for $i=\sum_j (r,n_j,y_1,y_2,\dots,y_{n_j},0,0,\dots)1_{A_j}$. This shows that $\mathcal{V}(x)=(V_i(x))_{i\in I}$
is a stable family of stable sets. In the following we frequently view $\mathcal{V}(x)$ as a set of functions on $\Omega$ by identifying it with the stable set $I$.

The topology induced by the neighborhood base $\mathcal{V}(x)$ is referred to as a \emph{stable topology} and is denoted by $\sigma_s(L^0(X),L^0(Y))$ or simply by $\sigma_s$.
Stable topologies are introduced in \cite{DJKK13} within conditional set theory, we refer to \cite{JZ2017compact} for their connection to $(\epsilon,\lambda)$-topologies and $L^0$-topologies. In this topology a stable net $(x_{\alpha})$ in $L^0(X)$ converges to $x$ if and only if $|\langle x_\alpha - x, y\rangle|\to 0$ a.e.~for all $y\in L^0(Y)$. Moreover, a stable subset $C$ of $L^0(X)$ is closed if and only if $x\in C$ for every stable net $(x_\alpha)$ in $C$ which converges to $x$. 

\begin{definition}\label{d:lsc}
 A function $f\colon X\to \bar L^0$ is called \emph{$\sigma_s$-lower semi-continuous}, if $f(x)\leq \liminf_\alpha f_s(x_{\alpha})$ for every stable net $(x_{\alpha})$ in $L^0_s(X)$ which converges  to $x\in X$. 

 The function $f$ is said to be proper convex, if $f(x)>-\infty$ for all $x\in X$ and $f(x^\prime)\in L^0$ for some $x^\prime\in X$,
 as well as $f(\lambda x + (1-\lambda)\tilde{x})\leq \lambda f(x)+(1-\lambda)f(\tilde{x})$ for every $\lambda\in \mathbb{R}$ with $0\leq \lambda \leq 1$ and all $x,\tilde{x}\in X$. 
\end{definition}

\begin{proposition}\label{p:lsc}
For a function  $f\colon X \to \bar{L}^0$ the following properties are equivalent.

\begin{itemize}
 \item[(i)] $f$ is $\sigma_s$-lower semi-continuous.
 \item[(ii)] $f(x)=\sup_{V\in \mathcal{V}(x)}\inf \{f_s(\tilde{x})\colon \tilde{x}\in V\cap L^0_s(X)\}$ for all $x\in X$.
\item[(iii)] The sublevel set $\{x \in L^0_s(X) \colon f_s(x)\leq a\}$ is closed for each $a\in \bar{L}^0$. 
\end{itemize}
\end{proposition}
\begin{proof}
By stability the properties $(i)$ and $(ii)$ can equivalently be formulated for all $x\in L^0_s(X)$.

$(i)\Rightarrow(ii)$: Fix $x\in X$. Obviously, one has
\[d:=\sup_{V\in \mathcal{V}(x)}\inf \{f_s(\tilde{x})\colon \tilde{x}\in V\cap L^0_s(X)\}\le f(x).\]
On the other hand, fix $\varepsilon\in L^0_{++}$ and consider the stable family of stable sets
\[
\Big\{x\in  V\cap L^0_s(X) : \inf_{\tilde{x} \in V\cap L^0_s(X)} \arctan( f_s(\tilde{x}))+\varepsilon \ge \arctan(f_s(x)) \Big\}_{V\in \mathcal{V}(x)}.
\]
By Remark \ref{l:choice}, there exists a stable family $(x_V)_{V\in\mathcal{V}(x)}$ such that $x_V\in V\cap L^0_s(X)$ and
\[\inf_{\tilde{x} \in V\cap L^0_s(X)} \arctan( f_s(\tilde{x}))+\varepsilon \ge \arctan(f_s(x_V))\]
for all $V\in \mathcal{V}(x)$. Since the stable index set $\mathcal{V}(x)$ is ordered by reverse inclusion,
the family $(x_V)_{V\in\mathcal{V}(x)}$ is a stable net which by construction converges to $x$.
By $\sigma_s$-lower semi-continuity of $f$ it follows that 
\[\arctan(d)\ge\liminf_{V\in \mathcal{V}(x)} \arctan(f_s(x_{V}))-\varepsilon \geq \arctan(f(x))-\varepsilon.\]

$(ii)\Rightarrow(iii)$: Let $a\in \bar L^0$ and $(x_\alpha)$ be a stable net with $f_s(x_\alpha)\leq a$ which converges to $x\in L^0_s(x)$. Then one has
\[
f_s(x)=\sup_{V\in \mathcal{V}(x)}\inf \{f_s(\tilde{x})\colon \tilde{x}\in V\cap L^0_s(X)\}\leq \liminf_{\alpha} f_s(x_\alpha)\leq a. 
\]

$(iii)\Rightarrow(i)$: Let $(x_{\alpha})$ be a stable net in $L^0_s(X)$ which converges to $x\in X$.
By Remark \ref{l:choice}, there exists for every $\varepsilon\in L^0_{++}$ a stable subnet $(x_\beta)$ of $(x_\alpha)$ such that 
\[\arctan(f_s(x_\beta))\le  \liminf_{\alpha} \arctan(f_s(x_\alpha))+\varepsilon\] for all $\beta$. Since $x_\beta\to x$ it follows that $\arctan(f(x))\le \liminf_{\alpha} \arctan(f_s(x_\alpha))+\varepsilon$, showing that $f$ is $\sigma_s$-lower semi-continuous.
\end{proof}

\begin{remark}
Let $X=Y=L^2$ on $(0,1]$ endowed with the Lebesgue measure on its Borel $\sigma$-algebra, and consider the identity map $\id:L^2\to L^2$.
Although the sublevel set $\{x\in L^2: \id(x)\le a\}$ is $\sigma(L^2,L^2)$-closed for each $a\in \bar L^0$, the identity map $\id$ is not $\sigma_s$-lower semi-continuous.

Indeed, fix $V^{r}_{(y_m)_{1\leq m \leq n}}(0)\in\mathcal{V}(0)$, and notice that
\[\sup_{1\le m\le n}|\langle x,y_m\rangle|\leq \| x\|\sup_{1\le m\le n}\|y_m\|\quad\mbox{for all } x\in L^0_s(L^2)\] by \eqref{eq:cauchyschwartz}. We can assume that $r/\sup_{1\le m\le n}\|y_m\|\ge \tilde{r}$ for a constant $\tilde{r}>0$, otherwise we partition $(0,1]=\bigcup_k A_k$
and carry out the following argument on each $A_k$. Then, there exists a sequence $(c_l)$ in $\mathbb{R}$ which converges to $+\infty$ such that
\[ x_l=\sum_{k=1}^l x^k_l 1_{(\frac{k-1}{l},\frac{k}{l}]}\in V^{r}_{(y_m)_{1\leq m \leq n}}(0)\cap L^0_s(L^2)\quad\mbox{for all }l\in\mathbb{N},\]
where $x_l^k=-c_l 1_{(\frac{k-1}{l},\frac{k}{l}]}\in L^2$ for all $1\le k\le l$. On the other hand, since $\id_s(x_l)=-c_l$ for all $l\in\mathbb{N}$, it follows that \[\id(0)=0 > -\infty=\sup_{V\in \mathcal{V}(0)}\inf \big\{\id_s(\tilde{x})\colon \tilde{x}\in V\cap L^0_s(L^2)\big\}.\] 
In particular, the identity $\id$ does not have an extension in the sense of Theorem \ref{t:lscextension} below. 
\end{remark}

\subsection{Extension result} Our goal is to extend a $\sigma_s$-lower semi-continuous function $f:X\to \bar L^0$ to a stable function
$F:L^0(X)\to \bar L^0$. We need the following definitions. 

\begin{definition}\label{d:lsc}
A function $F:L^0(X)\to \bar{L}^0$ is called 
\begin{itemize}
\item[(i)] \emph{stable}, if $F(\sum_k x_k 1_{A_k})=\sum_k F(x_k) 1_{A_k}$ for every sequence $(x_k)$ in $L^0(X)$ and each partition $(A_k)$ of $\Omega$,
\item[(ii)] \emph{$\sigma_s$-lower semi-continuous}, if $F(x)\leq \liminf_\alpha F(x_{\alpha})$ for every stable net $(x_{\alpha})$ in $L^0(X)$ converging to $x\in L^0(X)$,
\item[(iii)] \emph{$L^0$-linear}, if $F$ is $L^0$-valued and $F(\lambda x+ \tilde{x})=\lambda F(x)+ F(\tilde{x})$ for all $x,\tilde{x}\in L^0(X)$ and $\lambda\in L^0$,
\item[(iv)] \emph{$L^0$-proper convex}, if $F(x)>-\infty$ for all $x\in L^0(X)$ and $F(x^\prime)\in L^0$ for some $x^\prime\in L^0(X)$, as well as 
$F(\lambda x + (1-\lambda)\tilde{x})\leq \lambda F(x)+(1-\lambda)F(\tilde{x})$ for every $\lambda\in L^0$ with $0\leq \lambda \leq 1$ and all $x,\tilde{x}\in L^0(X)$.
\end{itemize}
\end{definition}

Next, we state the main extension result. 

\begin{theorem}\label{t:lscextension}
For every  $\sigma_s$-lower semi-continuous function  $f\colon X \to \bar{L}^0$
there exists a stable, $\sigma_s$-lower semi-continuous function $F\colon L^0(X) \to \bar{L}^0$ which satisfies $F|_X=f$. 
    
Moreover, if $f$ is proper convex, then this extension can be chosen $L^0$-proper convex.
\end{theorem}

\begin{proof}
 Define
    \begin{equation*}
        F(x):=\sup_{V\in \mathcal{V}(x)}\inf\{f_s(\tilde{x})\colon \tilde{x} \in V\cap L^0_s(X)\},\quad x\in L^0(X).
    \end{equation*}
  Notice that for every $x\in L^0(X)$ and $V^{r}_{(y_m)_{1\leq m \leq n}}(x)\in\mathcal{V}(x)$
it follows from \eqref{eq:cauchyschwartz} that
\[\sup_{1\le m\le n}|\langle x_k-x,y_m\rangle|\leq \|x_k-x\|\sup_{1\le m\le n}\|y_m\|\to 0\]
for every sequence $(x_k)$ in $L^0_s(X)$ such that $x_k\to x$ a.e., which shows that 
\[V^{r}_{(y_m)_{1\leq m \leq n}}(x)\cap L^0_s(X)\neq\emptyset.\] 
Hence, $F$ is a well-defined stable function since $f_s$ is a stable function on the stable set $V\cap L^0_s(X)$. Moreover, it follows from Proposition \ref{p:lsc} that $F$ is an extension of $f$. 
That $F$ satisfies the desired properties is shown in the following two steps.

    {\it Step 1.} We show that $F$ is $\sigma_s$-lower semi-continuous. 
            Fix $x\in L^0(X)$ and $\varepsilon\in L^0_{++}$. 
            There exists $V=V^{r}_{(y_m)_{1\leq m\leq n}}(x)$ in $\mathcal{V}(x)$ such that
            \begin{equation*}
                \arctan(F(x)) -\varepsilon\leq \inf\{\arctan(f_s(\tilde{x}))\colon \tilde{x} \in V\cap L^0_s(X)\}.
            \end{equation*}
            Fix $z\in V^{r/2}_{(y_m)_{1\le m\le n}}(x)$. For $\tilde{V}=V^{r/2}_{(y_m)_{1\leq m\leq n}}(z)\in \mathcal{V}(z)$
            it follows from the triangle inequality that $\tilde{V}\subseteq V$.
            Since $\tilde{V}\cap L^0_s(X)\subseteq V\cap L^0_s(X)$, we obtain
            \begin{align*}
                \arctan(F(x)) -\varepsilon \leq \inf\{\arctan(f_s(\tilde{z}))\colon \tilde{z} \in \tilde{V}\cap L^0_s(X)\} 
            \end{align*}
            so that
            \begin{align*}
                \arctan(F(x)) -\varepsilon\leq\arctan(F(z)). 
            \end{align*}
            This shows that for every $\varepsilon\in L^0_{++}$ there exists $V^\varepsilon\in \mathcal{V}(x)$ such that $\arctan(F(x)) -\varepsilon \leq \arctan(F(z))$ for all $z\in V^\varepsilon$. Hence 
            \begin{equation*}
                \arctan(F(x)) - \varepsilon\leq  \sup_{V\in \mathcal{V}(x)}\inf\{\arctan(F(z))\colon z \in V\}.
            \end{equation*}
            By letting $\varepsilon\downarrow 0$, and since 
            $\sup_{V\in \mathcal{V}(x)}\inf\{\arctan(F(z))\colon z\in V\}\leq \arctan(F(x))$ is trivially satisfied, it follows from the strict monotonicity of $\arctan$ that 
            \begin{equation*}
                F(x) = \sup_{V\in \mathcal{V}(x)}\inf\{F(z)\colon z \in V\}.
            \end{equation*}
In particular, $F(x)\leq \liminf_\alpha F(x_{\alpha})$ for every stable net $(x_{\alpha})$ in $L^0(X)$ which converges to $x\in L^0(X)$.

       {\it Step 2.} We show that $F$ is $L^0$-proper convex when $f$ is proper convex.
      Since $F$ is an extension of $f$ there exists $x^\prime\in X\subset L^0(X)$ such that $F(x^\prime)=f(x^\prime)\in L^0$.

      Note that for every $\lambda\in L^0_s(\mathbb{R})$ with $0\leq \lambda \leq 1$ it follows from convexity
            \begin{equation*}
                f_s(\lambda x + (1-\lambda) z)\leq \lambda f_s(x) + (1-\lambda) f_s(z),\quad \mbox{for all }x,z\in L^0_s(X).
\end{equation*}
Fix $x, z\in L^0(X)$ and $\lambda \in L^0_s(\mathbb{R})$ with $0\leq \lambda \leq 1$. Let
       \[V=V^{r}_{(y_m)_{1\leq m \leq n}}(\lambda x+(1-\lambda)z)\in \mathcal{V}(\lambda x+(1-\lambda)z),\]
       \[W=V^{r}_{(y_m)_{1\leq m \leq n}}(x)\in \mathcal{V}(x),\quad\mbox{and}\quad W^\prime=V^{r}_{(y_m)_{1\leq m \leq n}}(z)\in \mathcal{V}(z).\] 
For $\tilde{x}\in W\cap L^0_s(X)$ and $\tilde{z}\in W^\prime\cap L^0_s(X)$ it follows that
$\lambda \tilde{x} + (1-\lambda)\tilde{z}$ is in $V\cap L^0_s(X)$, which shows that
            \begin{multline*}
                \inf\{f_s(\tilde{x})\colon \tilde{x}\in V\cap L^0_s(X)\}
                \leq \inf\{f_s(\lambda \tilde{x} + (1-\lambda)\tilde{z})\colon \tilde{x} \in W\cap L^0_s(X), \, \tilde{z} \in W^\prime\cap L^0_s(X)\}\\
                \leq \inf\{\lambda f_s(\tilde{x})+(1-\lambda)f_s(\tilde{z})\colon \tilde{x} \in W\cap L^0_s(X),\, \tilde{z}\in W^\prime\cap L^0_s(X)\}\\
                =\lambda \inf\{ f_s(\tilde{x})\colon \tilde{x}\in W\cap L^0_s(X)\} + (1-\lambda) \inf\{f_s(\tilde{z})\colon \tilde{z}\in W^\prime\cap L^0_s(X)\},
            \end{multline*}
            where we employ the convention $-\infty+\infty=+\infty$. 
            Hence, for every $V\in \mathcal{V}(\lambda x+(1-\lambda)z)$ there exist $W\in\mathcal{V}(x)$ and $W^\prime\in \mathcal{V}(z)$ such that 
            \begin{multline*}
                \inf\{f_s(\tilde{x})\colon \tilde{x} \in V\cap L^0_s(X)\}\\ \leq \lambda \inf\{f_s(\tilde{x})\colon \tilde{x}\in W\cap L^0_s(X)\} + (1-\lambda) \inf\{f_s(\tilde{z})\colon \tilde{z} \in W^\prime\cap L^0_s(X)\}. 
            \end{multline*}
            By taking the supremum on both sides of the previous inequality, one obtains 
            \begin{equation*}
                F(\lambda x + (1-\lambda)z)\leq \lambda F(x) + (1-\lambda)F(z).  
            \end{equation*}
            The last inequality also holds for $\lambda\in L^0$ with $0\leq \lambda \leq 1$ by approximating $\lambda$ with step functions in $L^0_s(\mathbb{R})$ and using the $\sigma_s$-lower semi-continuity of $F$. 

           By way of contradiction, suppose there exist $x\in L^0(X)$ and $A\in \mathcal{F}$ with $\mu(A)>0$ such that $F(x)=-\infty$ on $A$. 
           Let $x_0\in L^0(X)$ with $F(x_0)\in L^0$. 
           From $L^0$-convexity we have $F(\lambda x_0 + (1-\lambda)x)=-\infty$ on $A$ for all $0\leq \lambda <1$. 
           Since $\lambda x_0 + (1-\lambda) x$ converges to $x_0$ as $\lambda$ tends to $1$, it follows from $\sigma_s$-lower semi-continuity that $F(x_0)=-\infty$ on $A$ which is a contradiction. 
\end{proof}

\section{Fenchel-Moreau type duality for vector-valued functions}\label{s:main}
We consider the setting of the previous section. 
Let $(X,Y,\langle\cdot,\cdot \rangle)$ be a dual pair of Banach spaces such that 
$|\langle x,y\rangle|\leq \|x\| \|y\|$ for all $x\in X$ and $y\in Y$, and both norm-closed unit balls are weakly closed. 
Recall that $\langle\cdot,\cdot\rangle$ extends to $L^0(X)\times L^0(Y)$ with values in $L^0$, and satisfies $|\langle x,y\rangle|\leq \|x\|\|y\|$ for all $x\in L^0(X)$ and $y\in L^0(Y)$. 
The next result shows that $(L^0(X),L^0(Y),\langle\cdot,\cdot \rangle)$ is an $L^0$-dual pair. 
\begin{lemma}
The functions 
\[y\mapsto\langle x,\cdot\rangle : L^0(Y)\to L^0\quad \mbox{and}\quad  x\mapsto\langle \cdot,y\rangle : L^0(X)\to L^0 \] 
are stable and $L^0$-linear for all $x\in L^0(X)$ and $y\in L^0(Y)$. 

Moreover, for every $x\in L^0(X)$ with $\mu(x=0)=0$ there exists $y\in L^0(Y)$ such that $\mu(\langle x,y\rangle=0)=0$, and symmetrically,  for every $y\in L^0(Y)$ with $\mu(y=0)=0$ there exists $x\in L^0(X)$ such that $\mu(\langle x,y\rangle=0)=0$.
\end{lemma}
\begin{proof}
We only show the separation argument. To that end, fix $x \in L^0(X)$ with $\mu(x=0)=0$. 
Then there exists $r\in L^0_{++}$ with $r=\sum_k r_k1_{A_k}$ for a sequence $(r_k)$ in $(0,\infty)$ and a partition $(A_k)$ of $\Omega$ such that $01_A\notin C_{3r}(x)1_A$
for all $A\in \mathcal{F}$ with $\mu(A)>0$, where $C_{3r}(x):=\{\tilde{x} \in L^0(X)\colon \|\tilde{x}-x\|\leq 3r\}$. 
Further, there exists $x^s\in L^0_s(X)$ such that $\|x-x^s\|\le r$. 
By the triangle inequality, $01_A \not\in C_{2r}(x^s)1_A$ for all $A\in\mathcal{F}$ with $\mu(A)>0$. 
We can assume that $x^s$ and $r$ are defined on the same partition, i.e.~$x^s=\sum_k x^s_k 1_{A_k}$, by changing if necessary to a common refinement. Then it holds 
\begin{equation*}
C_{2r}(x^s)=\left\{\tilde{x} \in L^0(X)\colon \tilde{x} 1_{A_k} \in C_{2r_k}(x^s_k)1_{A_k} \text{ for all } k\right\}. 
\end{equation*}
Since $C^X_{2r_k}(x_k^s):=\{\tilde x \in X\colon \| \tilde x - x_k^s\| \leq 2r_k\}$ is $\sigma(X,Y)$-closed in $X$, by strong separation there exist $y_k \in Y \setminus\{0\}$ and a constant $\delta_k>0$ such that 
\[
\inf_{\tilde{x}\in C^X_{2r_k}(x_k^s)} \langle \tilde{x}, y_k\rangle \geq \delta_k>0
\] 
for all $k$. 
It follows that 
\begin{equation}\label{eq:separation}
\inf_{\tilde{x} \in C_{r}(x^s)}\langle \tilde x, y\rangle \geq \delta >0, 
\end{equation}
where $y:=\sum_k y_k1_{A_k}$ and $\delta:=\sum_k \delta_k1_{A_k}$. 
Indeed, let $\tilde{x}\in C_{r}(x^s)$ and $(\tilde{x}_n)$ be a stable sequence in $C_{r}(\tilde{x}) \cap L^0_s(X)$ with $\|\tilde{x}_n-\tilde{x}\|\to 0$ a.e.. 
We have $\|\tilde{x}_{l}^n-x^s_{k}\|\leq 2r_k$ whenever $\mu(A_k\cap B_{l}^n)>0$, where $\tilde{x}_n=\sum_l \tilde{x}_{l}^n1_{B_{l}^{n}}$. 
From the stability of the extended duality pairing $\langle \cdot, \cdot \rangle$ we obtain 
\[
 \langle \tilde{x}_n, y\rangle=\sum_{k,l} \langle \tilde{x}_{l}^{n},y_k\rangle 1_{A_k\cap B_{l}^{n}} \geq \delta>0, 
\]
so that
\[
\langle \tilde{x}, y\rangle=\langle \tilde{x}- \tilde{x}_n, y\rangle+ \langle \tilde{x}_n, y\rangle \geq \langle \tilde{x}- \tilde{x}_n, y\rangle+\delta \rightarrow \delta>0, 
\]
which shows \eqref{eq:separation}. 
Since $x\in  C_{r}(x^s)$ we conclude $\mu(\langle x, y\rangle=0)=0$. 
\end{proof}

In view of the previous result conditional functional analysis becomes applicable. 
By an adaptation of the classical results for dual pairs, it follows from the conditional fundamental theorem of 
duality (see e.g.~\cite[Corollarly 4.48]{martindiss} in the setting of conditional set theory; for an adaptation to the present $L^0$-setting, see \cite[Proposition 6.6]{JZ2017compact}) that
every $\sigma_s(L^0(X),L^0(Y))$-continuous, $L^0$-linear function $h:L^0(X)\to L^0$ is of the form
$ \langle\cdot,y\rangle$ for some $y\in L^0(Y)$. 
In particular, the  $L^0$-dual space of $(L^0(X),\sigma_s(L^0(X),L^0(Y)))$ can be identified with $L^0(Y)$. 
As a consequence, an application of a conditional version of the Fenchel-Moreau theorem, see e.g.~\cite[Theorem 6.3.]{JZ2017compact}, yields that every $L^0$-proper convex, stable, $\sigma_s$-lower semi-continuous function $F:L^0(X)\to \bar L^0$
has the dual representation
\begin{equation}\label{condFM}
F(x)=\sup_{y\in L^0(Y)} \{ \langle x, y\rangle - F^\ast(y)\}, \quad x\in L^0(X), 
\end{equation}
for the $L^0$-convex conjugate $F^\ast(y):=\sup_{x\in L^0(X)} \{ \langle x, y\rangle - F(x)\}$ for all $y\in L^0(Y)$. 
%To apply \cite[Theorem 3.8]{kupper03}, we observe that $(|\langle \cdot,y\rangle|)_{y\in L^0(Y)}$ is a stable family of $L^0$-seminorms on $L^0(X)$ which induces an $L^0$-convex neighborhood base of $0\in L^0(X)$, see \cite[Definition 2.3]{kupper03}. 
%Moreover, this neighborhood base is stable. 
%In view of \cite[Lemma 2.18]{kupper03}, the topology induced in this way on $L^0(X)$ coincides with $\sigma_s(L^0(X),L^0(Y))$. 
%For a discussion of the connection between $L^0$-topologies and stable topologies and their relation with conditional locally convex topological vector spaces, we refer to \cite{JZ2017compact}. 

Now we are ready to state our main result.

\begin{theorem}\label{t:main}
Let $f\colon X\to \bar L^0$ be a proper convex function. 
Then $f$ is $\sigma_s$-lower semi-continuous if and only if it has the representation 
\begin{equation}\label{dualrep}
 f(x)=\sup_{y \in L^0(Y)} \left\{\langle x, y \rangle - f^\ast(y)\right\} \quad\mbox{for all } x\in X, 
 \end{equation}
 where $f^\ast\colon L^0(Y)\to \bar{L}^0$ is given by 
 \[
 f^\ast(y):=\sup_{x\in X} \left\{\langle x, y \rangle - f(x)\right\}. 
 \]
\end{theorem}

\begin{proof}
    Suppose that $f$ is $\sigma_s$-lower semi-continuous. 
    It follows from Theorem \ref{t:lscextension} that there exists an $L^0$-proper convex $\sigma_s$-lower semi-continuous extension $F\colon L^0(X)\to \bar{L}^0$ . 
    By \eqref{condFM}, one obtains 
    \begin{equation*}
        F(x)=\sup_{y \in L^0(Y)}\{\langle x,y\rangle - F^\ast(y)\}, \quad x\in L^0(X), 
    \end{equation*}
    for the $L^0$-convex conjugate 
    \[
    F^\ast(y)=\sup_{x\in L^0(X)} \{\langle x,y\rangle - F(x)\}, \quad y\in L^0(Y). 
    \]
    Since 
    \begin{equation}\label{eq19}
        f^\ast(y)=\sup_{x\in X} \{\langle x,y\rangle - f(x)\}\leq \sup_{x\in L^0(X)} \{\langle x,y\rangle - F(x)\}=F^\ast(y)
    \end{equation}
    and $\langle x,y\rangle - f^\ast(y)\leq f(x)$ for all $y\in L^0(Y)$, one has
    \begin{equation*}
        f(x)=F(x)=\sup_{y\in L^0(Y)}\{\langle x,y\rangle - F^\ast(y)\} \leq \sup_{y\in L^0(Y)}\{\langle x,y\rangle - f^\ast(y)\} \leq f(x) 
    \end{equation*}
    for all $x\in X$. 
    
    Conversely, suppose that $f$ satisfies the representation \eqref{dualrep}.
    Let $(x_\alpha)$ be a stable net in $L^0_s(X)$ which converges to $x$. Since $\tilde x\mapsto \langle\tilde x,y\rangle$ is $\sigma_s$-lower semi-continuous on $L^0_s(X)$ for all $y\in L^0(Y)$, it follows that
    \begin{align*}
f(x)&=\sup_{y\in L^0(Y)}\left\{\langle x,y\rangle-f^\ast(y)\right\}\\
&\le \liminf_\alpha \sup_{y\in L^0(Y)}\left\{\langle x_\alpha,y\rangle-f^\ast(y)\right\}\\
&= \liminf_\alpha f_s(x_\alpha), 
    \end{align*}
where in the last equality we used that \eqref{dualrep} also holds for $f_s$ on $L^0_s(X)$ by stability.
This shows that $f$ is $\sigma_s$-lower semi-continuous.
\end{proof}

\begin{remark}\label{r:maximality}
Let $f:X\to\bar L^0$ be a proper convex, $\sigma_s$-lower semi-continuous function. Then
    the extension $F$ in Theorem \ref{t:lscextension} is \emph{maximal} in the sense that $G\leq F$ for every $L^0$-proper convex, $\sigma_s$-lower semi-continuous extension $G$ of $f$. 
    In fact, we have $F^\ast \leq G^\ast$ by the same argumentation as in \eqref{eq19}, and therefore 
    \begin{align}
        F(x)=\sup_{y\in L^0(Y)}\{\langle x,y\rangle - F^\ast(y)\}\geq \sup_{y\in L^0(Y)}\{\langle x,y\rangle - G^\ast(y)\} =G(x), 
    \end{align}
    where the last equality follows from \eqref{condFM}. 
\end{remark}

%In the finite dimensional case, the representation result simplifies as follows. 

%\begin{example}
%Suppose that $\Omega=\{1,\dots,d\}$ for some $d\in\mathbb{N}$, and identify
%$\bar L^0$ with $(\bar{\mathbb{R}})^d$. 
% Let $f\colon X\to (\bar{\mathbb{R}})^d$ be proper convex $\sigma_s$-lower semi-continuous function, i.e.~$f_i:X\to\bar{\mathbb{R}}$
% is proper convex and $\sigma(X,Y)$-lower semi-continuous for all $1\le i\le d$. By Theorem \ref{t:main}, we obtain the representation
% \begin{equation*}
%  f(x)=\sup_{y\in Y^d}\{(\langle x,y_1\rangle, \ldots, \langle x, y_d\rangle) - f^\ast(y)\},
% \end{equation*}
%where $f^\ast(y)=\sup_{x\in X}\{(\langle x,y_1\rangle, \ldots, \langle x, y_d\rangle) - f(x)\}$ for all $y\in Y^d$. 
%\end{example}

\end{document}